\documentclass[a4paper,11pt]{article}
\usepackage[utf8]{inputenc}
\usepackage[T1]{fontenc}
\usepackage{lmodern}
\usepackage{microtype}
\usepackage[a4paper,textwidth=16cm,textheight=22cm,centering]{geometry}

\usepackage{mathtools}
\usepackage{amssymb,amsfonts}
\usepackage{amsthm}
\usepackage{setspace}
\usepackage{indentfirst}
\usepackage{enumitem}
\usepackage{float}
\usepackage{tikz}

\usepackage{cite}
\usepackage{xcolor}
\usepackage[ ]{hyperref}
\hypersetup{
	colorlinks=true,
	linktoc=all,
	linkcolor=red,
	citecolor=blue,
	urlcolor=black,
	pdftitle={Packing and Covering Cycles Through Prescribed Vertices},
	pdfauthor={Hanzhi Bai and Jin Yan}}

\theoremstyle{plain}
\newtheorem{theorem}{Theorem}[section]
\newtheorem{lemma}[theorem]{Lemma}

\newtheorem{question}[theorem]{Question}

\theoremstyle{definition}
\newtheorem{definition}[theorem]{Definition}

\theoremstyle{remark}

\newcommand{\cC}{\mathcal{C}}
\newcommand{\cQ}{\mathcal{Q}}

\begin{document}
	\date{}
	\begin{spacing}{1}

\title{Packing and Covering Cycles Through Prescribed Vertices}

\author{
Hanzhi Bai$^{1}$\and
Jin Yan$^{1}$\thanks{Corresponding author. E-mail: {\tt yanj@sdu.edu.cn.} (J.~Yan)
}}

\footnotetext[1]{{School of Mathematics, Shandong University, Jinan 250100, P.~R.~China.
Emails: {\tt bhz@mail.sdu.edu.cn} (H.~Bai).
Supported by the National Natural Science Foundation of China (Grant No.~12571373) and the Natural Science Foundation of Shandong Province (Grant No.~ZR2025MS05).}}
\maketitle

\begin{abstract}
Let $G$ be a finite simple graph and let $S\subseteq V(G)$.
We prove that the minimum number of vertices meeting every cycle that
intersects $S$ is at most the maximum number of vertices of $S$ covered by a
collection of vertex-disjoint cycles.  This answers a question posed by Bowler,
Ghorbani, Gut, Jacobs, and Reich [\emph{SIAM Journal on Discrete Mathematics}
\textbf{40} (2026), 988--999].  An incidence-based reduction to their
bidirected packing--covering theorem preserves the packing value and projects
transversals without increasing their cardinality.
\end{abstract}



\section{Introduction}

Throughout the paper, graphs and digraphs are finite, and graphs are simple
and undirected; we denote them by $G$ and $D$, respectively.  A
\emph{bidirected graph} $B$ is obtained by assigning a sign $+$ or $-$ to each
end of every edge of a finite loopless multigraph.

A feedback vertex set is a vertex set whose deletion leaves an acyclic graph;
Karp included the feedback vertex set problem for digraphs in his 1972 list
\cite{Karp1972}.  The
fundamental connection between feedback vertex sets and cycle packings is the
theorem of Erd\H{o}s and P\'osa
\cite{ErdosPosa1965}: there exists an absolute constant $c$ such that every
graph either contains $k$ vertex-disjoint cycles or has a set of at most
$ck\log k$ vertices meeting every cycle.  This initiated the study of the
Erd\H{o}s--P\'osa property; see Raymond and Thilikos
\cite{RaymondThilikos2017} for a survey.

A natural refinement prescribes a set $S\subseteq V(G)$ and considers only
cycles containing a vertex of $S$, usually called $S$-cycles.  Kakimura,
Kawarabayashi, and Marx \cite{KakimuraKawarabayashiMarx2011} established an
Erd\H{o}s--P\'osa theorem for $S$-cycles with the explicit transversal bound
$40k^{2}\log_2 k$, and Pontecorvi and Wollan \cite{PontecorviWollan2012}
later obtained the classical $O(k\log k)$ order.  These results compare the
number of vertex-disjoint $S$-cycles with the minimum size of an $S$-cycle
transversal, rather than the number of prescribed vertices covered.

Cycles in digraphs have a different history.  Reed, Robertson, Seymour, and
Thomas \cite{ReedRobertsonSeymourThomas1996} proved Younger's conjecture that
cycles in digraphs satisfy an Erd\H{o}s--P\'osa-type packing--covering duality.
For cycles in digraphs through a prescribed set, Kakimura and Kawarabayashi
\cite{KakimuraKawarabayashi2012} established a bounded-fractionality result,
and Kawarabayashi, Kr\'al', Kr\v{c}\'al, and Kreutzer
\cite{KawarabayashiKralKrcalKreutzer2013} obtained a half-integral analogue,
in which each vertex belongs to at most two selected cycles.  The latter work
also shows that the corresponding pairwise vertex-disjoint packing statement
fails in general.

Bowler, Ghorbani, Gut, Jacobs, and Reich \cite{BowlerEtAl2026} studied the number of prescribed vertices or edges covered by packed cycles. If $H$ is
a graph, digraph, or bidirected graph, a cycle of $H$ is understood in the
corresponding sense.  For
$S\subseteq V(H)$, let $\nu^{\mathrm v}_S(H)$ be the maximum number of
vertices of $S$ covered by pairwise vertex-disjoint cycles, and let
$\tau^{\mathrm v}_S(H)$ be the minimum size of a set $X\subseteq V(H)$ that
meets every cycle $C$ with $V(C)\cap S\neq\varnothing$.  For
$F\subseteq E(H)$, let $\nu^{\mathrm e}_F(H)$ be the maximum number of edges
of $F$ covered by pairwise edge-disjoint cycles, and let
$\tau^{\mathrm e}_F(H)$ be the minimum size of a set $Y\subseteq E(H)$ such
that $Y\cap E(C)\neq\varnothing$ for every cycle $C$ with
$E(C)\cap F\neq\varnothing$.

The following theorem collects Theorems~2.1, 2.2, 2.3, and 2.7 and Proposition~2.8 in \cite{BowlerEtAl2026}.
\begin{theorem}\label{thm:bowler-results}
The following statements hold.
\begin{enumerate}[label=\textnormal{(\roman*)},leftmargin=2.4em,itemsep=1pt,
                  topsep=3pt]
\item If $G$ is a graph and $F\subseteq E(G)$, then
      $\tau^{\mathrm e}_F(G)\leq\nu^{\mathrm e}_F(G)$
      \cite[Theorem~2.1]{BowlerEtAl2026}.
\item If $D$ is a digraph and $F\subseteq E(D)$, then
      $\tau^{\mathrm e}_F(D)\leq\nu^{\mathrm e}_F(D)$
      \cite[Theorem~2.2]{BowlerEtAl2026}.
\item If $B$ is a bidirected graph and $S\subseteq V(B)$, then
      $\tau^{\mathrm v}_S(B)\leq\nu^{\mathrm v}_S(B)$
      \cite[Theorem~2.3]{BowlerEtAl2026}.
\item If $D$ is a digraph and $S\subseteq V(D)$, then
      $\tau^{\mathrm v}_S(D)\leq\nu^{\mathrm v}_S(D)$
      \cite[Theorem~2.7]{BowlerEtAl2026}.
\item For every positive integer $k$, there exist a bidirected graph
      $B_k$ and $F_k\subseteq E(B_k)$ such that
      $\nu^{\mathrm e}_{F_k}(B_k)\leq 1$ and
      $\tau^{\mathrm e}_{F_k}(B_k)>k$
      \cite[Proposition~2.8]{BowlerEtAl2026}.
\end{enumerate}
\end{theorem}

The vertex version for graphs remained open.  Replacing each edge by two
oppositely oriented arcs does not resolve it, because the resulting digraph
contains $2$-cycles absent from the original graph.  For a graph $G$, write
$\nu_S(G):=\nu^{\mathrm v}_S(G)$ and $\tau_S(G):=\tau^{\mathrm v}_S(G)$. In \cite{BowlerEtAl2026}, the authors asked the following question.

\begin{question}\label{conj:bowler}
For every graph $G$ and every $S\subseteq V(G)$,
\[
  \tau_S(G)\leq \nu_S(G).
\]
\end{question}

We resolve this question in the following theorem.

\begin{theorem}\label{thm:main}
For every graph $G$ and every $S\subseteq V(G)$,
\[
  \tau_S(G)\leq \nu_S(G).
\]
\end{theorem}

The proof uses the bidirected theorem of Bowler et al.  Each vertex and edge
of $G$ is represented by an auxiliary pair.  For each incidence $(v,e)$, an
auxiliary vertex joins the pair representing $v$ to the pair representing
$e$.  The assigned signs force bidirected cycles to alternate
between vertex and edge pairs, giving a cycle correspondence that preserves
vertex-disjointness and the number of covered vertices of $S$.  A vertex map
then projects auxiliary transversals without increasing their size.  The
gadget for an edge $e=uv$ appears in Figure~\ref{fig:edge-gadget}.

Section~\ref{sec:prelim} fixes terminology and records the bidirected theorem.
Section~\ref{sec:construction} constructs the auxiliary graph and proves the
cycle correspondence; Section~\ref{sec:proof} completes the proof of
Theorem~\ref{thm:main}.

\section{Terminology and Preliminary Results}\label{sec:prelim}

We use the terminology of Diestel \cite{Diestel2017} and the conventions
stated in the Introduction.  A multigraph may have parallel edges.  If
$e=uv$, then $u$ and $v$ are its \emph{endvertices}.  A vertex $v$ and an
edge $e$ are \emph{incident} if $v$ is an endvertex of $e$; the pair $(v,e)$
is then an \emph{incidence}.  For a multigraph $H$, let $I(H)$ denote its set
of incidences.

For any collection $\cC$ of cycles, set
$V(\cC):=\bigcup_{C\in\cC}V(C)$ and
$E(\cC):=\bigcup_{C\in\cC}E(C)$.  Two cycles are
\emph{vertex-disjoint}, respectively \emph{edge-disjoint}, if their vertex
sets, respectively edge sets, are disjoint.  A \emph{cycle packing} is a
collection of pairwise vertex-disjoint cycles.  A vertex set $X$ \emph{meets}
a subgraph $H$ if $X\cap V(H)\neq\varnothing$.  Whenever
$x_0,\ldots,x_r$ occur consecutively on a path in this order,
$x_0x_1\cdots x_r$ denotes the corresponding subpath; its edges are
$x_{j-1}x_j$ for $1\leq j\leq r$.

A family $\mathcal{F}$ of graphs has the \emph{vertex Erd\H{o}s--P\'osa
property} if there is a function $f$ such that, for every graph $G$ and every
positive integer $k$, either $G$ contains $k$ pairwise vertex-disjoint
subgraphs isomorphic to members of $\mathcal{F}$, or a set of at most $f(k)$
vertices of $G$ meets every such subgraph.

\begin{definition}\label{def:parameters}
Let $G$ be a graph and let $S\subseteq V(G)$.  An \emph{$S$-cycle} is a cycle
$C$ satisfying $V(C)\cap S\neq\varnothing$.  An \emph{$S$-cycle transversal}
is a set $X\subseteq V(G)$ that meets every $S$-cycle; when $S=V(G)$, we call
$X$ a \emph{cycle transversal}.  Let $\nu_S(G)$ be the maximum of
$|S\cap V(\cC)|$ over all cycle packings $\cC$ in $G$, and let $\tau_S(G)$
be the minimum cardinality of an $S$-cycle transversal.  The empty packing is
permitted.  Hence both parameters are zero when $G$ has no $S$-cycle.
\end{definition}
We use the standard bidirected graph convention described in
\cite{BowlerEtAlMenger2023,BowlerEtAl2026}.

\begin{definition}\label{def:bidirected}
A \emph{bidirected graph} $B=(H,\sigma)$ consists of a loopless multigraph
$H$, called its \emph{underlying multigraph}, and a signing $\sigma$ that
assigns a sign $\sigma(x,f)\in\{+,-\}$ to every \emph{half-edge} $(x,f)$,
where $(x,f)\in I(H)$.  A cycle of $H$ is a \emph{bidirected cycle} if, at
each of its vertices, the two half-edges belonging to the cycle have opposite
signs; two parallel edges may form a cycle of length two.  The terms cycle packing, $S$-cycle,
$S$-cycle transversal, $\nu_S$, and $\tau_S$ then refer only to bidirected
cycles.
\end{definition}

In Section~\ref{sec:construction}, $z^0,z^1$ are the \emph{object vertices}
associated with $z$, $\{z^0,z^1\}$ its \emph{associated pair}, and
$b_z=z^0z^1$ its \emph{internal edge}; for each $(v,e)\in I(G)$,
$p_{v,e}$ is the \emph{incidence vertex} associated with $(v,e)$.  An
associated pair is a \emph{vertex pair}
for $z\in V(G)$ and an \emph{edge pair} for $z\in E(G)$.
We shall use \cite[Theorem~2.3]{BowlerEtAl2026}, stated below.  It is the only
external theorem used in our proof of Theorem~\ref{thm:main}.

\begin{lemma}\label{lem:bidirected}
For every bidirected graph $B$ and every $T\subseteq V(B)$, we have
$\tau_T(B)\leq \nu_T(B)$.
\end{lemma}

\section{The Auxiliary Bidirected Graph}\label{sec:construction}

Fix a graph $G$ and a set $S\subseteq V(G)$.  We construct a
bidirected graph $B$ from the vertex--edge incidence structure of $G$.  For
each object $z\in V(G)\cup E(G)$, introduce two vertices $z^0,z^1$ joined by
the internal edge $b_z:=z^0z^1$.  For every incidence $(v,e)\in I(G)$,
introduce a vertex $p_{v,e}$.  For each $i\in\{0,1\}$, add the two edges
$p_{v,e}v^i$ and $p_{v,e}e^i$, called \emph{incidence edges}.  Thus
\[
\begin{aligned}
V(B)&=\{z^0,z^1:z\in V(G)\cup E(G)\}
      \cup\{p_{v,e}:(v,e)\in I(G)\},\\
E(B)&=\{b_z:z\in V(G)\cup E(G)\}
      \cup\{p_{v,e}v^i,p_{v,e}e^i:(v,e)\in I(G),\ i\in\{0,1\}\}.
\end{aligned}
\]
Writing $n=|V(G)|$ and $m=|E(G)|$, we have
$|V(B)|=2n+4m$ and $|E(B)|=n+9m$.  Since $G$ has neither loops nor parallel
edges and each incidence has its own vertex, the underlying multigraph of
$B$ also has neither loops nor parallel edges.

Assign the half-edge signs by
\[
\begin{aligned}
  \sigma(z^0,b_z)&=\sigma(z^1,b_z)=-,\\
  \sigma(p_{v,e},p_{v,e}v^i)&=\sigma(v^i,p_{v,e}v^i)=+,\\
  \sigma(p_{v,e},p_{v,e}e^i)&=-,\ \sigma(e^i,p_{v,e}e^i)=+.
\end{aligned}
\]
Thus, at each $z^i$, the half-edge of $b_z$ is the unique incident half-edge
with sign $-$; at $p_{v,e}$, the half-edges leading to the vertex pair have
sign $+$ and those leading to the edge pair have sign $-$.  Finally, set
$T:=\{v^0:v\in S\}\subseteq V(B)$.

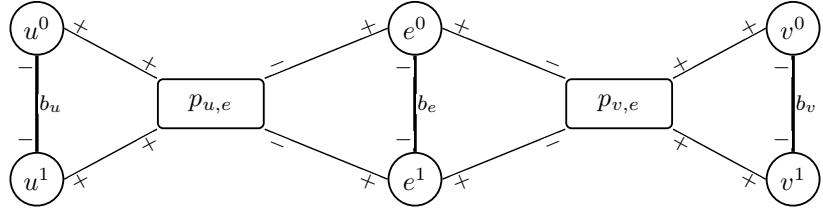
\begin{figure}[H]
\centering
\begin{tikzpicture}[
  x=1cm,
  y=1cm,
  object/.style={circle,draw,line width=.7pt,fill=white,minimum size=7mm,
                 inner sep=0pt,font=\small},
  incidence/.style={draw,rounded corners=2pt,line width=.7pt,fill=white,
                    minimum width=14mm,minimum height=6.5mm,inner sep=1pt,
                    font=\small},
  incidence edge/.style={line width=.55pt},
  distinguished edge/.style={line width=1.3pt},
  half sign/.style={font=\scriptsize,fill=white,inner sep=.35pt}
]
\node[object]    (u0) at (-5.0, 1.0) {$u^0$};
\node[object]    (u1) at (-5.0,-1.0) {$u^1$};
\node[incidence] (pu) at (-2.7, 0.0) {$p_{u,e}$};
\node[object]    (e0) at ( 0.0, 1.0) {$e^0$};
\node[object]    (e1) at ( 0.0,-1.0) {$e^1$};
\node[incidence] (pv) at ( 2.7, 0.0) {$p_{v,e}$};
\node[object]    (v0) at ( 5.0, 1.0) {$v^0$};
\node[object]    (v1) at ( 5.0,-1.0) {$v^1$};

\draw[distinguished edge]
  (u0.south) -- node[half sign,pos=.13,left] {$-$}
  node[half sign,pos=.87,left] {$-$}
  node[half sign,pos=.50,right] {$b_u$} (u1.north);
\draw[distinguished edge]
  (e0.south) -- node[half sign,pos=.13,left] {$-$}
  node[half sign,pos=.87,left] {$-$}
  node[half sign,pos=.50,right] {$b_e$} (e1.north);
\draw[distinguished edge]
  (v0.south) -- node[half sign,pos=.13,left] {$-$}
  node[half sign,pos=.87,left] {$-$}
  node[half sign,pos=.50,right] {$b_v$} (v1.north);

\draw[incidence edge]
  (u0.east) -- node[half sign,pos=.13,above,sloped] {$+$}
  node[half sign,pos=.87,above,sloped] {$+$} (pu.north west);
\draw[incidence edge]
  (u1.east) -- node[half sign,pos=.13,below,sloped] {$+$}
  node[half sign,pos=.87,below,sloped] {$+$} (pu.south west);
\draw[incidence edge]
  (pu.north east) -- node[half sign,pos=.13,above,sloped] {$-$}
  node[half sign,pos=.87,above,sloped] {$+$} (e0.west);
\draw[incidence edge]
  (pu.south east) -- node[half sign,pos=.13,below,sloped] {$-$}
  node[half sign,pos=.87,below,sloped] {$+$} (e1.west);
\draw[incidence edge]
  (e0.east) -- node[half sign,pos=.13,above,sloped] {$+$}
  node[half sign,pos=.87,above,sloped] {$-$} (pv.north west);
\draw[incidence edge]
  (e1.east) -- node[half sign,pos=.13,below,sloped] {$+$}
  node[half sign,pos=.87,below,sloped] {$-$} (pv.south west);
\draw[incidence edge]
  (pv.north east) -- node[half sign,pos=.13,above,sloped] {$+$}
  node[half sign,pos=.87,above,sloped] {$+$} (v0.west);
\draw[incidence edge]
  (pv.south east) -- node[half sign,pos=.13,below,sloped] {$+$}
  node[half sign,pos=.87,below,sloped] {$+$} (v1.west);
\end{tikzpicture}
\caption{The gadget for $e=uv$.  All edges introduced for the endvertices
$u$ and $v$ of $e$ are shown; other incidence edges at $u^i$ and $v^i$ are
omitted.  Signs label the
adjacent half-edges, and $b_u,b_e,b_v$ are thick.}
\label{fig:edge-gadget}
\end{figure}

\begin{lemma}\label{lem:block-traversal}
Let $Q$ be a bidirected cycle in $B$.  If $Q$ contains $z^i$ for some
$z\in V(G)\cup E(G)$ and $i\in\{0,1\}$, then $Q$ contains the subpath
$pz^iz^{1-i}q$, where $p$ and $q$ are distinct incidence vertices.  Moreover,
at every
incidence vertex $p_{v,e}$ on $Q$, one incident edge of $Q$ leads to
$v^0$ or $v^1$ and the other leads to $e^0$ or $e^1$.
\end{lemma}

\begin{proof}
At $z^i$, the edge $b_z$ is the unique incident edge with sign $-$.  Since the
two cycle edges at $z^i$ must have different signs, $Q$ uses $b_z$ and exactly
one edge joining $z^i$ to an incidence vertex $p$ whose half-edge at
$z^i$ has sign $+$.  Applying the same argument at the other endvertex
$z^{1-i}$ of $b_z$, the cycle uses exactly one incidence edge whose half-edge
at $z^{1-i}$ has sign $+$.  Let $q$ be the endvertex of this incidence edge
distinct from $z^{1-i}$.  Hence the indicated subpath lies on $Q$.

Suppose that $p=q$.  If $z\in V(G)$, then both edges of this subpath incident
with $p$ lead to the vertex pair associated with $z$ and consequently both
have sign $+$ at $p$.  If $z\in E(G)$, then both have sign $-$ at $p$.
In either case, the two edges of $Q$ incident with $p$ have the same
half-edge sign at $p$, contrary to Definition~\ref{def:bidirected}.  Thus
$p\neq q$.

Now consider an arbitrary incidence vertex $p_{v,e}$ on $Q$.  Every edge from
$p_{v,e}$ to $v^0$ or $v^1$ has sign $+$ at $p_{v,e}$, and every edge from
$p_{v,e}$ to $e^0$ or $e^1$ has sign $-$.  The two cycle edges at $p_{v,e}$
must have opposite signs.  Therefore one leads to the pair associated with
$v$ and the other leads to the pair associated with $e$.
\end{proof}

\begin{lemma}\label{lem:cycle-correspondence}
For a bidirected cycle $Q$ of $B$, let
$V_Q:=\{v\in V(G):\{v^0,v^1\}\subseteq V(Q)\}$ and
$E_Q:=\{e\in E(G):\{e^0,e^1\}\subseteq V(Q)\}$.  The subgraph of $G$ with
vertex set $V_Q$ and edge set $E_Q$ is a cycle, denoted by $\pi(Q)$.
Conversely, for every cycle $C$ of $G$, there is a bidirected cycle $Q$ of
$B$ satisfying $\pi(Q)=C$.  Moreover,
$|T\cap V(Q)|=|S\cap V(\pi(Q))|$.
\end{lemma}

\begin{proof}
Every edge of $B$ has an endvertex of the form $z^i$.  Hence every cycle of $B$
contains an object vertex.  By Lemma~\ref{lem:block-traversal}, whenever $Q$
contains one vertex in a pair associated with $z$, it contains both vertices
of that pair and the internal edge $b_z$.  At every incidence vertex, the same
lemma forces $Q$ to pass between a vertex pair and an edge pair.  Consequently,
the pairs encountered cyclically along $Q$ alternate between pairs associated
with vertices of $G$ and pairs associated with edges of $G$.

Let $e=uv\in E(G)$.  The pair $\{e^0,e^1\}$ is adjacent only to
$p_{u,e}$ and $p_{v,e}$.  If $\{e^0,e^1\}\subseteq V(Q)$, then
Lemma~\ref{lem:block-traversal} gives a subpath
$pe^ie^{1-i}q$ of $Q$ in which $p$ and $q$ are distinct incidence
vertices.  It follows that $\{p,q\}=\{p_{u,e},p_{v,e}\}$.  Applying the final
assertion of Lemma~\ref{lem:block-traversal} at $p_{u,e}$ and $p_{v,e}$ shows
that both vertex pairs associated with $u$ and $v$ occur on $Q$.

Similarly, if the vertex pair associated with $v$ occurs on $Q$, then
Lemma~\ref{lem:block-traversal} gives a subpath
$p_{v,e}v^iv^{1-i}p_{v,f}$ for two distinct edges $e$ and $f$ incident
with $v$.  Thus exactly two members of $E_Q$ are incident with each member of
$V_Q$.  Reading the associated pairs cyclically along $Q$ also shows that the
subgraph $(V_Q,E_Q)$ is connected.  Hence it is a connected $2$-regular
subgraph of $G$.  Since $G$ has neither loops nor parallel edges, it has at
least three vertices and is therefore a cycle.  The sets $V_Q$ and $E_Q$
determine it uniquely; this proves that $\pi(Q)$ is well defined.

For the converse, write $C=v_1e_1v_2e_2\cdots v_ke_kv_1$, where $k\geq3$,
$e_j=v_jv_{j+1}$, and all indices are taken modulo $k$; in particular,
$e_0=e_k$.  For each $j\in\{1,\ldots,k\}$, choose
$\alpha_j,\beta_j\in\{0,1\}$.  In cyclic order of $j$, concatenate the $2k$
paths $p_{v_j,e_{j-1}}v_j^{\alpha_j}v_j^{1-\alpha_j}p_{v_j,e_j}$ and
$p_{v_j,e_j}e_j^{\beta_j}e_j^{1-\beta_j}p_{v_{j+1},e_j}$.  Cyclically
consecutive paths have exactly their indicated common endvertex, while all
other vertices are distinct because $C$ is a cycle.  Their union is connected
and every vertex in it has degree two; it is therefore a cycle in the
underlying multigraph of $B$.  Denote this cycle by $Q_C$.

At each object vertex on this cycle, the half-edge belonging to the internal
edge has sign $-$, while the half-edge belonging to the incidence edge has
sign $+$.  At each incidence vertex $p_{v_j,e_j}$, one cycle edge leads to a
vertex pair and has half-edge sign $+$ there, while the other leads to an edge
pair and has half-edge sign $-$ there.  The two relevant half-edge signs are
therefore different at every vertex, so Definition~\ref{def:bidirected} shows
that $Q_C$ is bidirected.  Its associated vertex and edge sets are precisely
those of $C$, and hence $\pi(Q_C)=C$.

Finally, the first part of the proof shows that $Q$ contains $v^0$ if and only
if it contains the whole pair $\{v^0,v^1\}$, which holds if and only if
$\pi(Q)$ contains $v$.  Since $T$ contains precisely one auxiliary vertex
$v^0$ for each $v\in S$, summing this equivalence over $v\in S$ yields
$|T\cap V(Q)|=|S\cap V(\pi(Q))|$.
\end{proof}

\section{Packing and Transversals}\label{sec:proof}

We first show that the auxiliary construction preserves the packing value.

\begin{lemma}\label{lem:packing}
For the auxiliary bidirected graph $B$ and the set $T$ defined above, we have
$\nu_T(B)=\nu_S(G)$.
\end{lemma}

\begin{proof}
Let $\cQ$ be a collection of pairwise vertex-disjoint bidirected cycles in
$B$.  Suppose that two distinct cycles $Q_1,Q_2\in\cQ$ had projections that
shared a vertex $v\in V(G)$.  By Lemma~\ref{lem:cycle-correspondence}, both
$Q_1$ and $Q_2$ would contain $v^0$ and $v^1$, contradicting the
vertex-disjointness of $\cQ$.  Therefore
$\{\pi(Q):Q\in\cQ\}$ is a collection of pairwise vertex-disjoint cycles in
$G$.  The cycles in $\cQ$ are disjoint, and Lemma~\ref{lem:cycle-correspondence}
gives
\[
\begin{split}
  |T\cap V(\cQ)|
  &=\sum_{Q\in\cQ}|T\cap V(Q)|\\
  &=\sum_{Q\in\cQ}|S\cap V(\pi(Q))|
   =\left|S\cap V\bigl(\{\pi(Q):Q\in\cQ\}\bigr)\right|\\
  &\leq \nu_S(G).
\end{split}
\]
Taking the maximum over all admissible $\cQ$ proves
$\nu_T(B)\leq\nu_S(G)$.

Conversely, let $\cC$ be a collection of pairwise vertex-disjoint cycles in
$G$.  For every $C\in\cC$, choose a bidirected cycle $Q_C$
given by Lemma~\ref{lem:cycle-correspondence} such that $\pi(Q_C)=C$.  Two
cycles in $\cC$ share no vertex and hence share neither an edge nor an
incidence.  Therefore the cycles $Q_C$ use disjoint associated pairs and
disjoint incidence vertices, so they are pairwise vertex-disjoint in $B$.
The equality in Lemma~\ref{lem:cycle-correspondence} shows that the collection
$\{Q_C:C\in\cC\}$ covers exactly $|S\cap V(\cC)|$ vertices of $T$.
Maximizing over $\cC$ proves $\nu_T(B)\geq\nu_S(G)$.
\end{proof}

We next project an arbitrary transversal in $B$ to a transversal in $G$.

\begin{lemma}\label{lem:transversal}
If $X_B\subseteq V(B)$ meets every bidirected $T$-cycle in $B$, then there is
an $S$-cycle transversal $X_G\subseteq V(G)$ satisfying
$|X_G|\leq|X_B|$.
\end{lemma}

\begin{proof}
For every edge $e=uv\in E(G)$, fix one endvertex $r(e)\in\{u,v\}$.  Define
$\phi:V(B)\to V(G)$ by $\phi(v^i)=v$, $\phi(p_{v,e})=v$, and
$\phi(e^i)=r(e)$, where $v\in V(G)$, $e\in E(G)$, and $i\in\{0,1\}$.

We claim that $\phi(x)\in V(\pi(Q))$ for every bidirected cycle $Q$ in $B$
and every $x\in V(Q)$.  There are three cases.  If $x=v^i$, then
Lemma~\ref{lem:block-traversal} forces $Q$ to use the entire pair associated
with $v$, so $v\in V(\pi(Q))$.  If $x=p_{v,e}$, then the final assertion of
Lemma~\ref{lem:block-traversal} shows that one of the two cycle edges at
$p_{v,e}$ leads to $v^0$ or $v^1$.  Hence $Q$ uses the pair associated with
$v$, and again $v\in V(\pi(Q))$.  If $x=e^i$, then $Q$ uses the entire pair
associated with $e$.  The projected cycle consequently traverses $e$ and
contains both endvertices of $e$, including $r(e)$.  This proves the claim.

Set $X_G:=\phi(X_B)$.  Then $|X_G|=|\phi(X_B)|\leq|X_B|$.  Let $C$ be an arbitrary
$S$-cycle of $G$.  By Lemma~\ref{lem:cycle-correspondence}, there is a
bidirected cycle $Q$ in $B$ such that $\pi(Q)=C$.  Because $C$ contains a
vertex $v\in S$, the definition of $\pi(Q)$ implies that
$\{v^0,v^1\}\subseteq V(Q)$; in particular, $v^0\in T$.  Thus $Q$ is a
$T$-cycle.  The hypothesis on $X_B$
gives a vertex $x\in X_B\cap V(Q)$.  The preceding case analysis then yields
$\phi(x)\in X_G\cap V(\pi(Q))=X_G\cap V(C)$.  Hence $X_G$ meets $C$.  Since
$C$ was arbitrary, $X_G$ meets every $S$-cycle in $G$.
\end{proof}

\begin{proof}[Proof of Theorem~\ref{thm:main}]
Let $X_B$ be a minimum $T$-cycle transversal of $B$.  By
Definition~\ref{def:parameters} and Lemma~\ref{lem:bidirected},
$|X_B|=\tau_T(B)\leq\nu_T(B)$.
Lemma~\ref{lem:transversal} supplies an $S$-cycle transversal
$X_G\subseteq V(G)$ with $|X_G|\leq|X_B|$.  Using
Lemma~\ref{lem:packing}, we obtain
\[
  \tau_S(G)\leq |X_G|\leq |X_B|
  \leq\nu_T(B)=\nu_S(G).
\]
This is the asserted inequality.
\end{proof}

\paragraph{Declaration on the Use of Generative AI}

The authors used generative AI in including discussion of preliminary formulations and assistance in improving the clarity of the exposition. Any AI-generated suggestion was critically assessed, independently verified, and adopted only at the authors' discretion. The authors take full responsibility for the originality, correctness, and final form of this work.

\end{spacing}

\end{document}